\documentclass{article}

\usepackage[utf8]{inputenc} 
\usepackage[T1]{fontenc}    
\usepackage{hyperref}       
\usepackage{url}            
\usepackage{booktabs}       
\usepackage{amsfonts}       
\usepackage{nicefrac}       
\usepackage{microtype}      

\usepackage{times}
\usepackage{graphicx} 

\usepackage[square,sort,comma,numbers]{natbib}
\usepackage{algorithm}
\usepackage{algorithmic}

\usepackage{amsmath}
\usepackage{amssymb}
\usepackage{relsize}
\usepackage{dsfont}
\usepackage{amsthm}
\usepackage{caption}
\usepackage{subcaption}
\usepackage{bm}
\usepackage{refcount}
\usepackage{multirow}
\usepackage{setspace}
\usepackage{lipsum}

\newcommand{\Expmap}[2]{\ensuremath{\text{Exp}_{#1}(#2)}}
\newcommand{\Expmapinv}[2]{\ensuremath{\text{Exp}^{-1}_{#1}(#2)}}
\newcommand{\Tanspace}{\ensuremath{T_{x}\mathcal{M}}}
\newcommand{\Sqnorm}[1]{\ensuremath{\|#1\|^{2}}}

\DeclareMathOperator*{\argmin}{\arg\!\min}

\newtheorem{lemma}{Lemma}
\newtheorem{proposition}{Proposition}

\title{Accelerated Stochastic Quasi-Newton Optimization on Riemannian Manifolds}

\author{
  Anirban Roychowdhury\\
  Department of Computer Science and Engineering\\
  Ohio State University\\
  Columbus, OH 43210 \\
  \texttt{roychowdhury.7@osu.edu} \\
}

\begin{document}

\maketitle

\begin{abstract}
We propose an L-BFGS  optimization algorithm on Riemannian manifolds using minibatched stochastic variance reduction techniques for fast convergence with constant step sizes, without resorting to linesearch methods designed to satisfy Wolfe conditions. We provide a new convergence proof for strongly convex functions without using curvature conditions on the manifold, as well as a convergence discussion for nonconvex functions. We discuss a couple of ways to obtain the correction pairs used to calculate the product of the gradient with the inverse Hessian, and empirically demonstrate their use in synthetic experiments on computation of Karcher means for symmetric positive definite matrices and leading eigenvalues of large scale data matrices. We compare our method to VR-PCA for the latter experiment, along with Riemannian SVRG for both cases, and show strong convergence results for a range of datasets.
\end{abstract}

\section{Introduction} 

Optimization algorithms are a mainstay in machine learning research, underpinning solvers for a wide swath of problems ranging from linear regression and SVMs to deep learning. Consequently, scaling such algorithms to large scale datasets while preserving theoretical guarantees is of paramount importance. An important challenge in this field is designing scalable algorithms for optimization problems in the presence of constraints on the search space, a situation all too often encountered in real life.
One approach to handling such constrained optimization problems on vector spaces is to reformulate them as optimization tasks on a suitable Riemannian manifold, with the constraints acting as manifold parametrization. Often, the problems can be shown to possess desirable geometric properties like convexity with respect to distance-minimizing geodesics on the manifold, leading to provably efficient optimization algorithms \cite{absilbook, zhangcolt16, conicsuvrit, ring_wirth}. These ideas can then be combined with stochastic optimization techniques influenced by \cite{robbins_monro}, to deal with large datasets with theoretical convergence guarantees. See \cite{bonnabel, zhangnips16} for recent examples. For instance, we can consider the problem of computing leading eigenvectors in the PCA setting \cite{shamiricml15} with unit-norm constraints. Projection-based strategies are normally used for this kind of problems \cite{oja}, but alternating between solving and projecting can be prohibitively expensive in high dimensions. However, the unit-norm constraint can be used to cast the eigenvector problem into an unconstrained optimization scenario on the unit sphere, which happens to be one of the most well-behaved Riemannian manifolds.

Once the problems have been cast onto manifolds, one would want fast optimization algorithms that potentially use stochastic minibatches to deal with very large datasets. Such algorithms operating in Euclidean space have been widely researched in the optimization literature, but their development for Riemannian manifolds has been limited so far. In particular, one should note that the convergence speed limitations of unconstrained stochastic algorithms in the Euclidean case apply to manifold optimization as well; for instance a straightforward port of stochastic gradient descent to Riemannian manifolds \cite{zhangcolt16} attains the same sublinear convergence seen in Euclidean space. There has been extensive work in the Euclidean domain using variance-reduced gradients to address this issue, with the aim of improving convergence rates by explicitly reducing the variance of stochastic gradients with suitably spaced full-gradient evaluations \cite{shamiricml15,rienips13}. Another nice advantage of this technique is the removal of the need for decaying learning rates for proving convergence, thereby solving the sublinearity issue as well sidestepping the nontrivial task of selecting an appropriate decay rate for SGD-like algorithms in large-scale optimization scenarios. Researchers have begun porting these methods to the manifold optimization domain, with a stochastic first-order variance reduced technique \cite{zhangnips16} showing robust convergence guarantees for both convex and nonconvex problems on geodesically complete manifolds.

Another complementary approach to improving convergence rates is of course using second-order updates for the iterates. In the Euclidean setting, one can show quadratic convergence rates for convex problems using Newton iterations, but these tend to be prohibitively expensive in high-dimensional big-data settings due to the need to store and invert the Hessian matrix.  
This limitation has led to the development of quasi-Newton methods, most notably L-BFGS \cite{liulbfgs}, which uses lower-order terms to approximate the inverse Hessian. The curvature information provided by the Hessian estimate allows superlinear convergence in ideal settings \cite{nocedal_wright}. While widely used for small-to-medium scale problems, adoption of these methods for big data problems has been limited, since the second order updates can be prohibitively expensive to compute in these situations. However, most optimization algorithms in the literature that use stochastic minibatching techniques to deal with large datasets are modifications of first order gradient-descent \cite{bottounips04,bottouiccs10} with relatively slower convergence in practical situations.
This has recently begun to be addressed, with researchers devising stochastic variants of the L-BFGS technique \cite{byrdarxiv14}, with straightforward convergence analyses. This has also been combined with variance reduction techniques and shown to have a linear convergence rate for convex problems in Euclidean space \cite{moritzaist16}. Our work in this paper is in a similar vein: we study quasi-Newton L-BFGS updates with stochastic variance reduction techniques for optimization problems on Riemannian manifolds, and analyze their convergence behavior for convex and nonconvex functions.

\noindent {\bf Contributions:} The main contributions of this work may be summarized as follows:

\textbf{1.} We propose a stochastic L-BFGS method for Riemannian manifolds using stochastic variance reduction techniques for the first-order gradient estimates, and analyze the convergence for both convex and nonconvex functions under standard assumptions.

\textbf{2.} Our proof for strongly convex functions is different from those of recently proposed stochastic L-BFGS algorithms using variance-reduced gradients in Euclidean space \cite{moritzaist16} due to different bounds on the stochastic gradients. We do not use sectional curvature bounds in our proof for the convex case, making it structurally different from that of Riemannian SVRG \cite{zhangnips16}. 

\textbf{3.} We show strong experimental results on Karcher mean computations and calculation of leading eigenvalues, with noticeably better performance than Riemannian SVRG and VR-PCA; the latter is one of the best performing Euclidean algorithms for the finding dominant eigenvalues, that also uses stochastic variance-reduced gradients.


\section{Preliminaries} 
\label{prelims}
\subsection{Riemannian geometry} \label{prelimone} We begin with a brief overview of the differential geometric concepts we use in  this work. We consider $C^{\infty}$ (smooth) manifolds that are locally homeomorphic to open subsets of $\mathbb{R}^{D}$, in the sense that the neighborhood of each point can be assigned a system of coordinates of appropriate dimensionality. Formally, this is defined with the notion of a \textit{chart} $c:U\rightarrow\mathbb{R}^{D}$ at each $x\in\mathcal{M}$, where $U\subset\mathcal{M}$ is an open subspace containing $x$. Smooth manifolds are ones with covering collections of differentiable ($C^{\infty}$) charts. A \textit{Riemannian metric} $g(\cdot,\cdot)$ is a bilinear $C^{\infty}$ tensor field of type $0\choose 2$, that is also symmetric and positive definite. A manifold endowed with such a metric is called a Riemannian manifold. The tangent space $T_{x}\mathcal{M}$ at every $x\in\mathcal{M}$ is a vector space, with the Riemannian metric $g:T_{x}\mathcal{M}\times T_{x}\mathcal{M}\rightarrow\mathbb{R}$ as the attendant metric. $g$ then induces a norm for vectors in the tangent space , which we denote by $\|\cdot\|$.

Riemannian manifolds are endowed with the Levi-Civita connection, which induces the notion of parallel transport of vectors from one tangent space to another along a geodesic, in a metric preserving way. That is, we have an operator $\Gamma_{\gamma}:T_{x}\mathcal{M}\rightarrow T_{y}\mathcal{M}$ where, informally speaking, $\gamma$ joins $x\text{ and }y$, and for any $u,\nu\in T_{x}\mathcal{M}$, we have $g(u,\nu)=g(\Gamma(u), \Gamma(\nu))$ . The parallel transport can be shown to be an isometry.

For every smooth curve $\gamma:[0,1]\rightarrow\mathcal{M}$ lying in $\mathcal{M}$, we denote its velocity vector as $\dot{\gamma}(t)\in T_{x}\mathcal{M}$ for each $t\in[0,1]$, with the ``speed'' given by $\|\dot{\gamma}(t)\|$. The length of such a curve is usually measured as $L(\gamma) = \int\limits_{0}^{1}\|\dot{\gamma}(t)\|dt.$ Denoting the covariant derivative along $\gamma$ of some $\nu\in T_{x}\mathcal{M}$, with respect to the Riemannian (Levi-Civita) connection by $A\nu$, we call $A\dot{\gamma}$ the \textit{acceleration} of the curve. Curves with constant velocities ($A\dot{\gamma}\equiv 0$) are called \textit{geodesics}, and can be shown to generalize the notion of Euclidean straight lines. We assume that every pair $x,y\in\mathcal{M}$ can be connected by a geodesic $\gamma$ s.t. $\gamma(0)=x\text{ and }\gamma(1)=y$. Immediately we have the notion of ``distance'' between any $x,y\in\mathcal{M}$ as the minimum length of all geodesics connecting $x\text{ and }y$, assuming the manifolds are \textit{geodesically complete} as mentioned above, in that every countable decreasing sequence of lengths of geodesics connecting a pair of points has a well-defined limit. 

The geodesic induces a useful operator called the \textit{exponential map}, defined as $\text{Exp}_{x}:T_{x}\mathcal{M}\rightarrow\mathcal{M}\text{ s.t. }\Expmap{x}{\nu}=\gamma(1)\text{ where }\gamma(0)=x,\gamma(1)=y\text{ and }\dot{\gamma}(0)=\nu$. If there is a unique geodesic connecting $x\text{ and }y$, then the exponential map has an inverse, denoted by $\Expmapinv{x}{y}$. The length of this geodesic can therefore be seen to be $\|\Expmapinv{x}{y}\|$.

The derivative $D$ of a differentiable function is defined using the Riemannian connection by the following equivalence: $Df(x)\nu=\nu f$, where $\nu\in T_{x}\mathcal{M}$. Then, by the Riesz representation theorem, there exists a gradient $\nabla f(x)\in T_{x}\mathcal{M}$ s.t. $\forall\nu\in T_{x}\mathcal{M},\>Df(x)\nu=g_{x}(\nabla f(x),\nu)$. Similarly, we can denote the Hessian as follows:  $D^{2}f(x)(\cdot,\cdot):T_{x}\mathcal{M}\times T_{x}\mathcal{M}\rightarrow\mathbb{R}$. We denote the mapping from $\nu\in T_{x}\mathcal{M}$ to the Riesz representation of $D^{2}f(x)(\nu,\cdot)$ by $\nabla^{2}f(x)$. One can consult standard textbooks on differential geometry \cite{leebook, boothby} for more details.

\subsection{Convexity and Lipschitz smoothness on manifolds} Similar to \cite{zhangcolt16,conicsuvrit,zhangnips16}, we define manifold (or geodesic) convexity concepts analogous to the Euclidean baselines, as follows : a set $U\subset\mathcal{M}$ is convex on the manifold if $\forall x,y\in U$ there exists a geodesic $\gamma$ connecting $x,y$ that completely lies in $U$, i.e. $\gamma(0)=x,\gamma(1)=y$. Then, a function can be defined as convex w.r.t. geodesics if $\forall x,y\in U\text{ where }\exists\gamma$ connecting $x,y$ on the manifold, we have:
\begin{equation*}
f(\gamma(t))\leq tf(x)+(1-t)f(y)\>\forall t\in[0,1].
\end{equation*}

We can also define a notion of strong convexity as follows: a function $f$ is called $S-$strongly convex if for any $x,y\in U$ and (sub)gradient $\nabla_{x}$, we have
\begin{equation}
\label{eqn:sconvex}
f(y)\geq f(x)+g_{x}(\nabla_{x},\Expmapinv{x}{y})+\frac{S}{2}\Sqnorm{\Expmapinv{x}{y}}.
\end{equation}

We define Lipschitz smoothness of a function $f$ by imposing Lipschitz continuity on the gradients, as follows: $\forall x,y\in U,$
\begin{equation*}
\|\nabla(x)-\Gamma_{\gamma}\nabla(y)\|\leq L\|\Expmapinv{x}{y}\|,
\end{equation*} where L is the smoothness parameter. Analogous to the Euclidean case, this property can also be formulated as::
\begin{equation}
\label{eqn:lsmooth}
f(y)\leq f(x)+g_{x}(\nabla_{x},\Expmapinv{x}{y})+\frac{L}{2}\Sqnorm{\Expmapinv{x}{y}}.
\end{equation}

\section{Stochastic Riemannian L-BFGS} 
In this section we present our stochastic variance-reduced L-BFGS algorithm on Riemannian manifolds and analyze the convergence behavior for convex and nonconvex differentiable functions on Riemannian manifolds. We assume these manifolds to be $L$-Lipschitz smooth, as defined above,  with existence of unique distance-minimizing geodesics between every two points, i.e. our manifolds are geodesically complete; this allows us to have a well-defined inverse exponential map that encodes the distance between a pair of points on the manifold. For the convergence analysis, we also assume $f$ to have a unique minimum at $x^{*}\in U$, where $U$ is a compact convex subset of the manifold. 

\subsection{The Algorithm} 
The pseudocode is shown in Algorithm~\ref{alg:srlbfgs}. We provide a brief discussion of the salient properties, and compare it to similar algorithms in the Euclidean domain, for example \cite{byrdarxiv14, moritzaist16}, as well as those on Riemannian manifolds, for example \cite{conicsuvrit}. To begin, note that $\nabla$ denotes the Riesz representation of the gradient $D$, as defined in $\S$\ref{prelimone}. We denote full gradients by $\nabla$ and stochastic gradients by $\tilde{\nabla}$. Similar to other stochastic algorithms with variance-reduction, we use two loops: each iteration of the inner loop corresponds to drawing one minibatch from the data and performing the stochastic gradient computations (Steps $10$, $11$), whereas each outer loop iteration corresponds to two passes over the full dataset, once to compute the full gradient (Step $6$) and the other to make multiple minibatch runs (Steps $8$ through $30$). Compared to the Euclidean setting, note that the computation of the variance-reduced gradient in Step $11$ involves an extra step: the gradients ($\nabla f(x)$-s) reside in the tangent spaces of the respective iterates, therefore we have to perform parallel transport to bring them to the same tangent space before performing linear combinations.

To avoid memory issues and complications arising from Hessian-vector computations in the Riemann setting, we chose to update the second correction variable $y_{r}$ using a simple difference of gradients approximation: $y_{r}=\tilde{\nabla}f(u_{r})-\Gamma_{\gamma}\tilde{\nabla}f(u_{r-1})$.
We should note here that the parallel transport parametrization should be clear from the context; $\Gamma_{\gamma}$ here denotes transporting the vector $\tilde{\nabla}f(u_{r-1})\in T_{u_{r-1}}\mathcal{M}$ to $T_{u_{r}}\mathcal{M}$ along the connecting geodesic $\gamma$. We omit any relevant annotations from the transport symbol to prevent notational overload. 
We calculate the first correction pair $z_{r}$ in one of two ways: (a) as $z_{r}=\Gamma_{\gamma}\left(\eta_{2}\rho_{t-1}\right)$, or (b) as $z_{r}=\Gamma_{\gamma}\left(-\eta_{1}\nu_{\text{prev}}\right)$. We denote these by \textbf{Option} $\mathbf{1}$ and \textbf{Option} $\mathbf{2}$ respectively in Alg. \ref{alg:srlbfgs}. Note that in both cases, $\Gamma_{\gamma}$ denotes the parallel transport of the argument to the tangent space at $x_{t}^{s+1}$. In our experiments, we noticed faster convergence for the strongly convex centroid computation problem with \textbf{Option} $\mathbf{1}$, along with computation of the correction pairs every iteration and a low memory pool. For calculating dominating eigenvalues on the unit-radius sphere, \textbf{Option} $\mathbf{2}$ yielded better results. Once the corrections pairs $z_{r}, y_{r}$ have been computed, we compute the descent step in Step 21 using the standard two-loop recursion formula given in \cite{nocedal_wright}, using the $M$ correction pairs stored in memory. Note that we use fixed stepsize in the update steps and in computing the correction pairs, and do not impose or perform calculations designed to have them satisfy Armijo or Wolfe conditions.
\begin{algorithm}[p]
   \caption{Riemannian Stochastic VR L-BFGS}
   \label{alg:srlbfgs}
   \begin{spacing}{1.2}
\begin{algorithmic}[1]\small
   \STATE {\bfseries Input:} Initial value $x^{0}$, parameters $M\text{ and }R$, learning rates $\eta_{1},\eta_{2}$, minibatch size $mb$.
   \STATE Initialize $c = 1$;
   \STATE Set $r=0$;
   \STATE Initialize $H_{0}$;
   \FOR{$t=0,1,\ldots$}
   \STATE Set $x_{0}^{t+1}=x^{s}$;
   \STATE Compute full gradient $g^{t+1}=N^{-1}\sum_{i=1}^{N}\nabla f_{i}(x^{t})$;
   \FOR{$i=0,1,\ldots m-1$}
   \STATE Sample minibatch $I_{i,mb}\subset{1,\ldots,N}$;
   \STATE Compute $\tilde{\nabla}f(x_{i}^{i+1})\text{ and }\tilde{\nabla}f(x^{i})$ using $I_{i,mb}$;
   \STATE Set $\nu_{i}^{t+1}=\tilde{\nabla}f(x_{i}^{t+1})-\Gamma_{\gamma}(\tilde{\nabla}f(x^{t})-g^{t+1})$;
   \IF{$c\equiv 0 \mod R$}
   \STATE Set $r=r+1$;
   \IF{$r\geq 2$}
   \STATE Set $u_{r}^{t+1}=x_{i}^{t+1}$;
   \STATE \textbf{Option 1}: Compute $z_{r}^{t+1}=\Gamma_{\gamma}\left(\eta_{2}\rho_{i-1}^{t+1}\right)$;
   \STATE \textbf{Option 2}: Compute $z_{r}^{t+1}=\Gamma_{\gamma}\left(-\eta_{1}\nu_{\text{prev}}\right)$;
   \STATE Compute $y_{r}^{t+1}=\tilde{\nabla}f(u_{r}^{t+1})-\Gamma_{\gamma}\tilde{\nabla}f(u_{i-1}^{t+1})$ using $I_{i,mb}$;
   \STATE Store correction pairs $z_{r}^{t+1}$ and $y_{r}^{t+1}$, using $r$ to maintain memory depth $M$;
   \ENDIF
   \STATE Set $x_{\text{prev}}=x_{i}^{t+1}$, $\nu_{\text{prev}}=\nu_{i}^{t+1}$;
   \ENDIF
   \IF{$c < 2R$}
   \STATE Set $x_{i+1}^{t+1}=\Expmap{x_{i}^{t+1}}{-\eta_{1}\nu_{i}^{t+1}}$;
   \ELSE
   \STATE Compute $\rho_{i}^{t+1}=H_{r}^{t+1}\nu_{i}^{t+1}$, as mentioned in the text;
   \STATE Set $x_{i+1}^{t+1}=\Expmap{x_{i}^{t+1}}{\eta_{2}\rho_{i}^{t+1}}$;
   \ENDIF
   \STATE Set $c=c+1$;
   \ENDFOR
   \STATE Set $x^{t+1}=x_{m}^{t+1}$;
   \ENDFOR
\end{algorithmic}
\end{spacing}
\end{algorithm}
Compared to the Euclidean algorithms \cite{byrdarxiv14, moritzaist16}, Alg.\ref{alg:srlbfgs} has some key differences: 1) we did not notice any significant advantage from using separate minibatches in Steps $10$ and $18$, therefore we use the same minibatch to compute the VR gradient and the correction elements $y_{r}$; 2) we do not keep a running average of the iterates for computing the correction element $z_{r}$ (Steps $15$ through $17$); 3) we use constant stepsizes throughout the whole process, in contrast to \cite{byrdarxiv14} that uses a decaying sequence. Note that, as seen in Step $24$, we use the first-order VR gradient to update the iterates for the first $2R$ iterations; this is because we calculate correction pairs every $R$ steps and evaluate the gradient-inverse Hessian product (Step $26$) once at least two pairs have been collected. Similar to \cite{nocedal_wright}, we drop the oldest pair to maintain the memory depth $M$. Compared to the algorithms in \cite{conicsuvrit, reshadnips15}, ours uses stochastic VR gradients, with all the attendant modifications and advantages, and does not use linesearch techniques to satisfy Wolfe conditions.

\subsection{Analysis of convergence} 
In this section we provide the main convergence results of the algorithm. We analyze convergence for finite-sum empirical risk minimization problems of the following form:
\begin{equation}
\label{eq:fsum}
\min_{x\in \mathcal{M}} f(x) = \frac{1}{N}\sum\limits_{i=1}^{N}f_{i}(x),
\end{equation} where the Riemannian manifold is denoted by $\mathcal{M}$. 
Note that the iterates are updated in Algorithm \ref{alg:srlbfgs} by taking the exponential map of the descent step multiplied by the stepsize, with the descent step computed as the product of the inverse Hessian estimate and the stochastic variance-reduced gradient using the standard two-loop recursion formula. Thus, to bound the optimization error using the iterates, we will need bounds on both the stochastic gradients and the inverse Hessians. As mentioned in \cite{zhangnips16}, the methods used to derive the former bounds for Euclidean algorithms cannot be ported directly to manifolds due to metric nonlinearities; see the proof of Proposition \ref{prop:one} for details. For the latter, we follow the standard template for L-BFGS algorithms in the literature \cite{ring_wirth,nocedal_wright,byrdarxiv14}. To begin, we make the following assumptions: 

\textbf{ Assumption 1.} The function $f$ in \eqref{eq:fsum} is strongly convex on the manifold, whereas the $f_{i}$s are individually convex.

\textbf{ Assumption 2.} There exist $\lambda,\Lambda\in (0,\infty),\>\lambda<\Lambda$ s.t. $\lambda\|\nu\|_{x}^{2}\leq D^{2}f\leq \Lambda\|\nu\|_{x}^{2}\quad\forall\nu\in T_{x}\mathcal{M}$.

These two assumptions allow us to (a) guarantee that $f$ has a unique minimizer $x^{*}$ in the convex sublevel set $U$, and (b) derive bounds on the inverse Hessian updates using BFGS update formulae for the Hessian approximations. Similar to the Euclidean case, these can be written as follows:
\begin{equation}
\label{eqn:bfgsone}
\hat{B}_{r}=\Gamma_{\gamma}\left[\hat{B}_{r-1}-\frac{B_{r-1}(s_{r-1},\cdot)\hat{B}_{r-1}s_{r-1}}{B_{r-1}(s_{r-1},s_{r-1})}\right]\Gamma_{\gamma}^{-1},
\end{equation} and by the Sherman-Morrison-Woodbury lemma, that of the inverse:
\begin{equation*}
H_{r} = \Gamma_{\gamma}\left[G^{-1}H_{r-1}G+\frac{g_{x_{r-1}}(s_{r-1},\cdot)s_{r-1}}{y_{r-1}s_{r-1}}\right]\Gamma_{\gamma}^{-1},
\end{equation*} where $G = I-\frac{g_{x_{r-1}}(s_{r-1},\cdot)\hat{y}_{r-1}}{y_{r-1}s_{r-1}}$, and $\hat{B}_{r}$ is the Lax-Milgram representation of the Hessian. Details on these constructs can be found in \cite{ring_wirth}, in addition to \cite{leebook,boothby}.

\subsubsection{Trace and determinant bounds}
To start off our convergence discussions for both convex and nonconvex cases, we derive bounds for the trace and determinants of the Hessian approximations, followed by those for their inverses. The techniques used to do so are straightforward ports of the Euclidean originals \cite{nocedal_wright}, with some minor modifications to account for differential geometric technicalities. Using the assumptions above, we can prove the following bounds \cite{ring_wirth}:
\begin{lemma}
\label{lem:trdet}
Let $B_{r}^{s+1}=\left(H_{r}^{s+1}\right)^{-1}$ be the approximation of the Hessian generated by Algorithm~\ref{alg:srlbfgs}, and $\hat{B}_{r}^{s+1}$ and $\hat{H}_{r}^{s+1}$ be the corresponding Lax-Milgram representations. Let $M$, the memory parameter, be the number of correction pairs used to update the inverse Hessian approximation. Then, under Assumptions 1 and 2, we have:
\begin{align*}
\mathrm{tr}(\hat{B}_{r}^{s+1})\leq \mathrm{tr}(\hat{B}_{0}^{s+1})+M\Lambda,\quad \det(\hat{B}_{r}^{s+1})&\geq \det(\hat{B}_{0}^{s+1})\frac{\lambda^{M}}{(\mathrm{tr}(\hat{B}_{0}^{s+1})+\Lambda M)^{M}}.
\end{align*} Also, $\gamma I\preceq \hat{H}_{r}^{s+1}\preceq \Gamma I $, for some $\Gamma\geq\gamma>0$.
\end{lemma} From a notational perspective, recall that our notation for the parallel transport operator is $\Gamma_{\gamma}$, with the subscript denoting the geodesic. The symbols $\gamma$ and $\Gamma$ in Lemma \ref{lem:trdet} above are unrelated to these geometric concepts, merely being the derived bounds on the eigenvalues of inverse Hessian approximations. The proof is given in the supplementary for completeness.

\subsubsection{Convergence result for strongly convex functions}
Our convergence result for strongly convex functions on the manifold can be stated as follows:
\begin{proposition}
\label{prop:one}
Let the Assumptions 1 and 2 hold. Further, let the $f(\cdot)$ in \eqref{eq:fsum} be S-strongly convex, and each of the $f_{i}$ be L-smooth, as defined earlier. Define the following constants: $p=\left[LS^{-1}+2\eta_{2}S^{-1}\left\lbrace2\eta L^{3}\Gamma^{2}-S\kappa\gamma\right\rbrace\right]$, and $q^{\prime}=6\eta^{2}L^{3}\Gamma^{2}S^{-1}$. Denote the global optimum by $x^{*}$. Then the iterate $x^{T+1}$ obtained after $T$ outer loop iterations will satisfy the following condition:
\begin{align*}
\mathbb{E}\left[f(x^{T+1})-f(x^{*})\right]\leq LS^{-1}\beta^{T}\mathbb{E}\left[f(x^{0})-f(x^{*})\right],
\end{align*} where the constants are chosen to satisfy $\beta=\left(1-p\right)^{-1}\left(q^{\prime}+p^{T}(1-p-q^{\prime})\right)<1$ for linear convergence.
\end{proposition}
For proving this statement, we will use the $L$-smoothness \eqref{eqn:lsmooth} and $S$-strong convexity \eqref{eqn:sconvex} conditions mentioned earlier. As in the Euclidean case \cite{rienips13, moritzaist16}, we will also require a bound on the stochastic variance-reduced gradients. These can be bounded using triangle inequalities and $L$-smoothness on Riemannian manifolds, as shown in \cite{zhangnips16}. This alternative is necessary since the Euclidean bound first derived in \cite{rienips13}, using the difference of the objective function at the iterates, cannot be ported directly to manifolds due to metric nonlinearities. Thus we take a different approach in our proof compared to the Euclidean case of \cite{moritzaist16}, using the interpoint distances defined with the norms of inverse exponential maps. We do not use trigonometric distance inequalities \cite{conicsuvrit, bonnabel} for the convex case either, making the overall structure different from the proof of Riemannian SVRG as well. The details are deferred to the supplementary due to space limitations. However we do use the trigonometric inequality along with assumed lower bounds on sectional curvature for showing convergence for nonconvex functions, as described next.

\subsubsection{Convergence for the nonconvex case}
Here we provide a convergence result for nonconvex functions satisfying the following condition: $f(x^{t})-f(x^{*})\leq \kappa^{-1}\|\nabla f(x^{t})\|^{2}$, which automatically holds for strongly convex functions. We assume this to hold even if $f$ is nonconvex, since it allows us to show convergence of the iterates using $\|\nabla f(x^{t})\|^{2}$. Further, similar to \cite{zhangcolt16, zhangnips16} we assume that the sectional curvature of the manifold is lower bounded by $c_{\delta}$. This allows us to derive a trigonometric inequality analogous to the Euclidean case, where the sides of the ``triangle'' are geodesics \cite{bonnabel}. The details are given in the supplementary. Additionally, we assume that the eigenvalues of the inverse Hessian are bounded by $(\gamma,\Gamma)$ within some suitable region around an optimum. The main result of this section may be stated as follows:
\begin{proposition}
Let the sectional curvature of the manifold be bounded below by $c_{\delta}$, and the $f_{i}$ be $L$-smooth. Let $x^{*}$ be an optimum of $f(\cdot)$ in \eqref{eq:fsum}. Assume the eigenvalues of the inverse Hessian estimates are bounded. Set $\eta_{2}=\mu_{0}/\left(\Gamma Ln^{\alpha_{1}}\eta^{\alpha_{2}}\right)$, $K=mT$, and $m=\lfloor n^{\nicefrac{3\alpha_{1}}{2}}/\left(3\mu_{0}\zeta^{1-2\alpha_{2}}\right)\rfloor$, where $\alpha_{1}\in(0,1]$ and $\alpha_{2}\in[0,2]$. Then, for suitable choices of the inverse Hessian bounds $\gamma,\Gamma$, we can find values for the constants $\mu_{0}>0$ and $\epsilon>0$ so that the following holds:
\begin{align*}
\mathbb{E}\|\nabla f(x^{T})\|^{2}\leq(K\epsilon)^{-1}L\eta_{2}^{\alpha_{1}}\zeta^{\alpha_{2}}\left(f(x^{0})-f(x^{*})\right).
\end{align*}
\end{proposition} $\zeta$ is defined as $\zeta=\left(\tanh\left(d\sqrt{|c_{\delta}|}\right)\right)^{-1}d\sqrt{|c_{\delta}|}$ if $c_{\delta}<0$, and $0$ otherwise; $d$ is an upper bound on the diameter of the set $U$ mentioned earlier, containing an optimum $x^{*}$.
The proof is inspired by similar results from both Euclidean \cite{reddinonconvex} and Riemannian \cite{zhangnips16} analyses, and is given in the supplementary. One way to deal with negative curvature in Hessians in Euclidean space is by adding some suitable positive $\alpha$ to the diagonal, ensuring bounds on the eigenvalues. Investigation of such ``damping'' methods in the Riemannian context could be an interesting area of future work.

\section{Experiments}
\subsection{Karcher mean computation for PD matrices}
We begin with a synthetic experiment on learning the Karcher mean (centroid) \cite{bhatiabook} of positive definite matrices. For a collection of matrices $\left\lbrace\mathbf{W}_{i}\right\rbrace_{i=1}^{N}$, the optimization problem can be stated as follows:
\begin{align*}
\argmin_{\mathbf{W}\succeq 0}\left\lbrace \sum\limits_{i=1}^{N}\|\log\left(\mathbf{W}^{-\nicefrac{1}{2}}X_{i}\mathbf{W}^{-\nicefrac{1}{2}}\right)\|_{F}^{2}\right\rbrace.
\end{align*}
We compare our minibatched implementation of the Riemannian SVRG algorithm from \cite{zhangnips16}, denoted as rSVRG, with the stochastic variance-reduced L-BFGS procedure from Algorithm \ref{alg:srlbfgs}, denoted rSV-LBFGS. We implemented both algorithms using the Manopt \cite{manopt} and Mixest \cite{mixest} toolkits. We generated three sets of random positive definite matrices, each of size $100\times 100$, with condition numbers $10, 1e2,$ and $1e3$, and computed the ground truths using code from \cite{binietal}. Matrix counts were $100$ for condition number $1e-2$, and $1000$ for the rest. Both algorithms used equal batchsizes of $50$ for the first and third datasets, and $5$ for the second, and were initialized identically. Both used learning rates satisfying their convergence conditions. In general we found rSV-LBFGS to perform better with frequent correction pair calculations and a low retention rate, ostensibly due to the strong convexity; therefore we used $R=1$, $M=2$ for all three datasets. As mentioned earlier, the $z_{r}$ correction pair was calculated using \textbf{Option 1}: $z_{r}=\Gamma_{\gamma}\left(\eta_{2}\rho_{t}\right)$ where $\rho_{t}$ is calculated using the two-loop recursion. We used standard retractions to approximate the exponential maps. The retraction formulae for both symmetric PD and sphere manifolds used in the next section are given in the supplementary.
\begin{figure*}[!tb]
\centering
\begin{subfigure}[b]{0.32\textwidth}
\includegraphics[width=\textwidth]{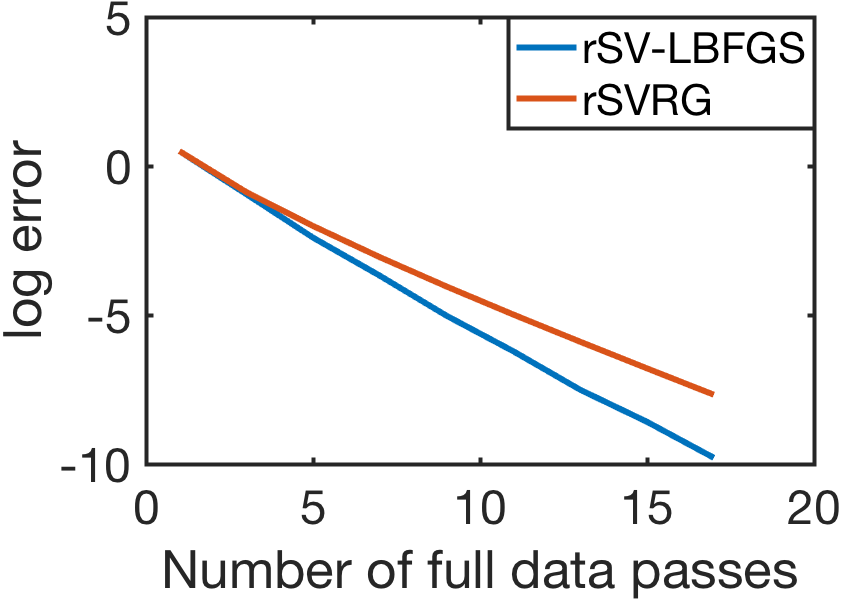}
\caption{ }
\label{figr:pdkm_1}
\end{subfigure}
\begin{subfigure}[b]{0.32\textwidth}
\includegraphics[width=\textwidth]{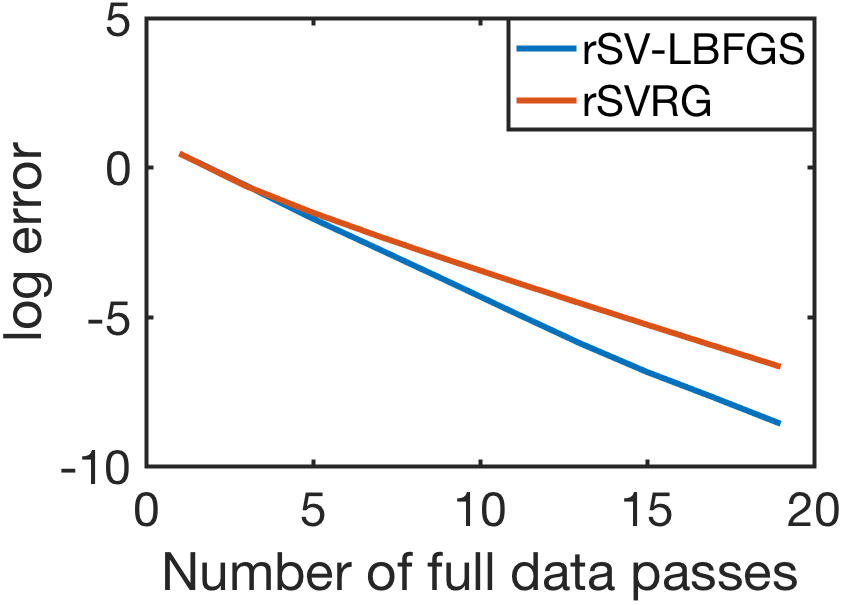}
\caption{ }
\label{figr:pdkm_2}
\end{subfigure}
\begin{subfigure}[b]{0.32\textwidth}
\includegraphics[width=\textwidth]{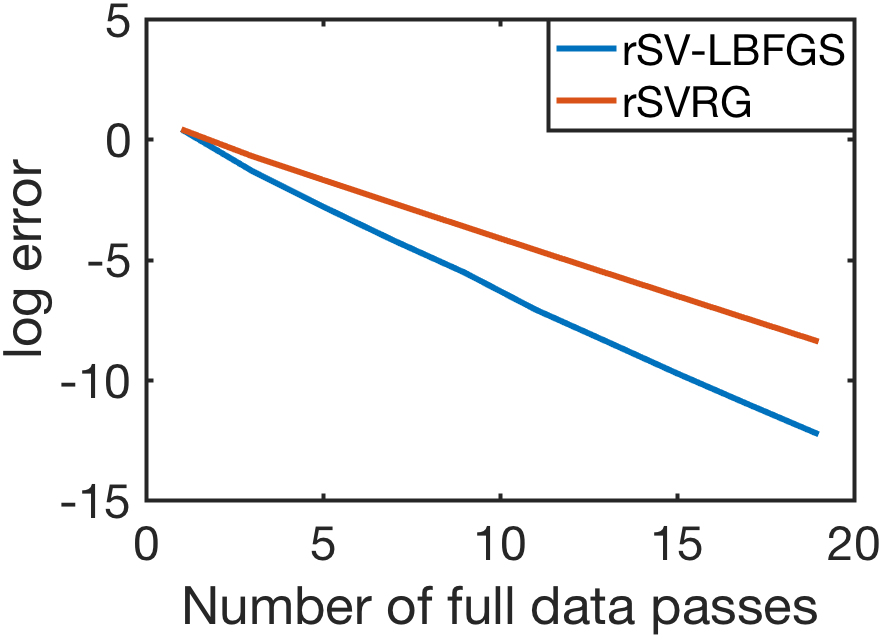}
\caption{ }
\label{figr:pdkm_3}
\end{subfigure}
\caption{Error decay plots for rSV-LBFGS and rSVRG obtained from the three synthetic Karcher mean computation experiments. Figures (a), (b) and (c) show the log-errors for datasets with condition numbers $1e3,1e2$ and $10$ respectively, plotted against number of passes over full dataset for each algorithm. See text for full details.}
\label{figpdkm}
\end{figure*}
We calculated the error of iterate $\mathbf{W}$ as $\|\mathbf{W}-\mathbf{W}^{*}\|_{F}^{2}$, with $\mathbf{W}^{*}$ being the ground truth. The log errors are plotted vs the number of data passes in Fig.\ref{figpdkm}. Comparing convergence speed in terms of \# data passes is often the preferred approach for benchmarking ML algorithms since it is an implementation-agnostic evaluation and focuses on the key bottleneck (I/O) for big data problems. Comparisons of rSVRG with Riemannian gradient descent methods, both batch and stochastic, can be found in \cite{zhangnips16}. From Fig.\ref{figpdkm}, we find rSV-LBFGS to converge faster than rSVRG for all three datasets.

\begin{figure*}[!tb]
\centering
\begin{subfigure}[b]{0.49\textwidth}
\includegraphics[width=\textwidth]{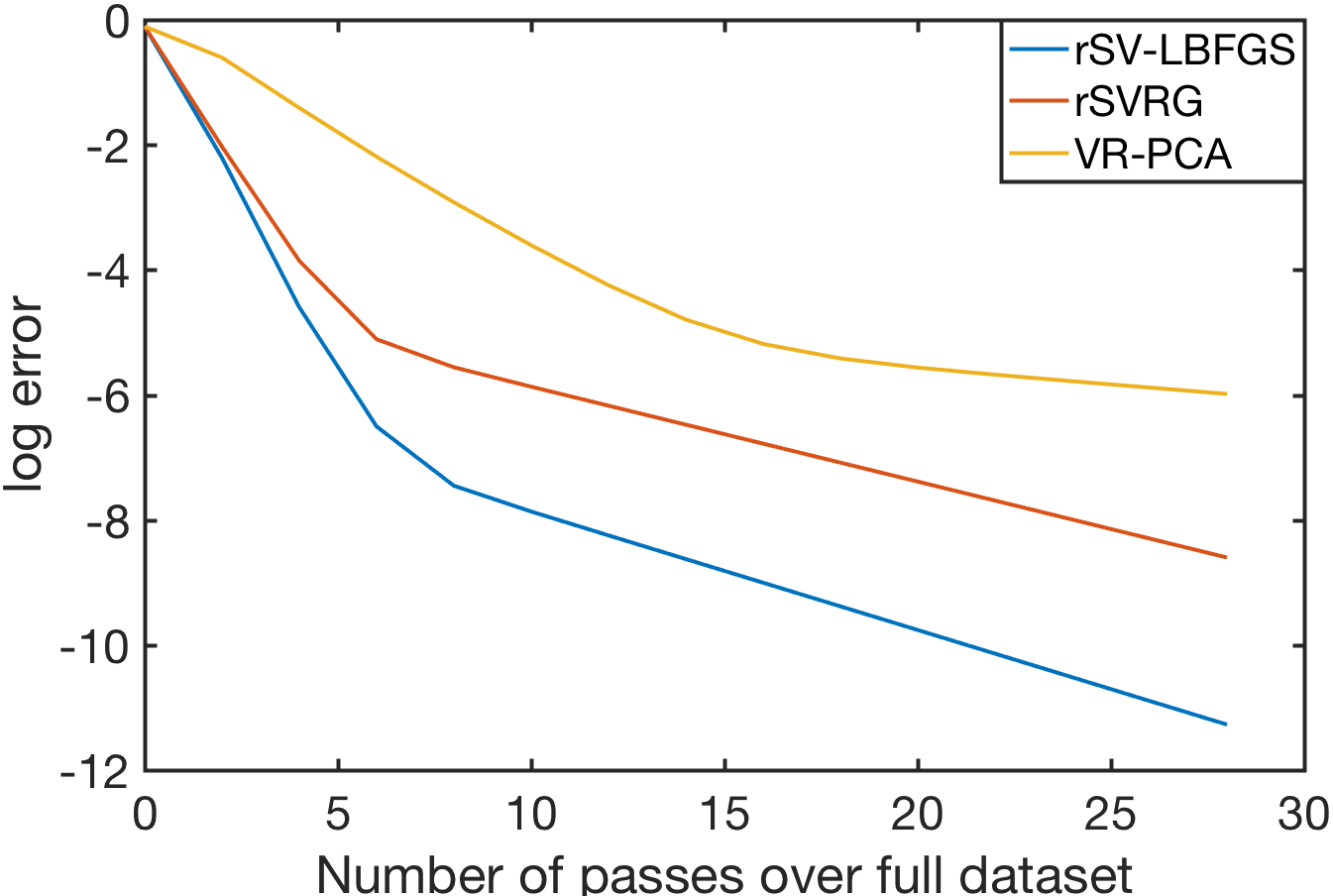}
\caption{ }
\label{fig:eigvdata1}
\end{subfigure}
\begin{subfigure}[b]{0.49\textwidth}
\includegraphics[width=\textwidth]{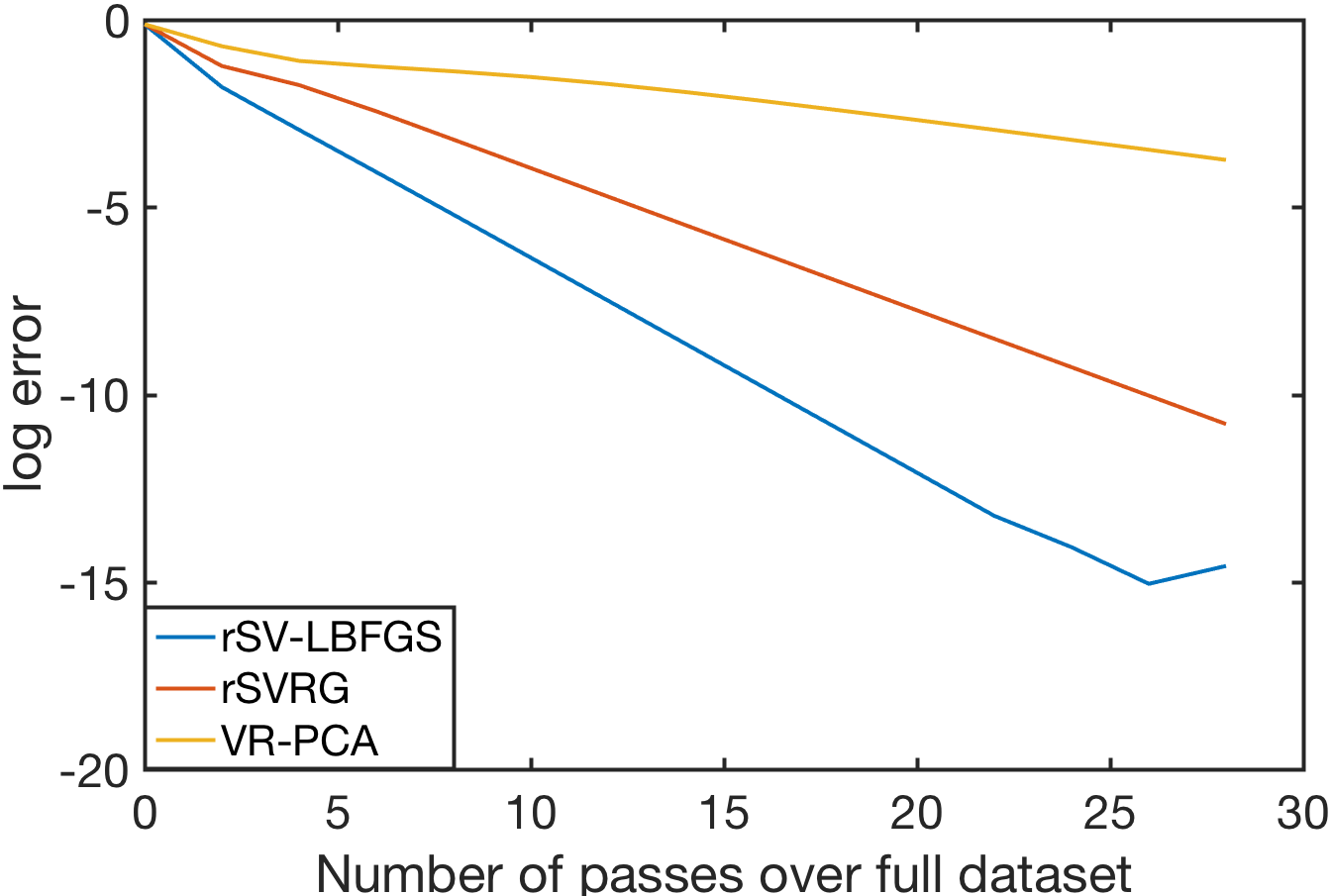}
\caption{ }
\label{fig:eigvdata2}
\end{subfigure}
\begin{subfigure}[b]{0.49\textwidth}
\includegraphics[width=\textwidth]{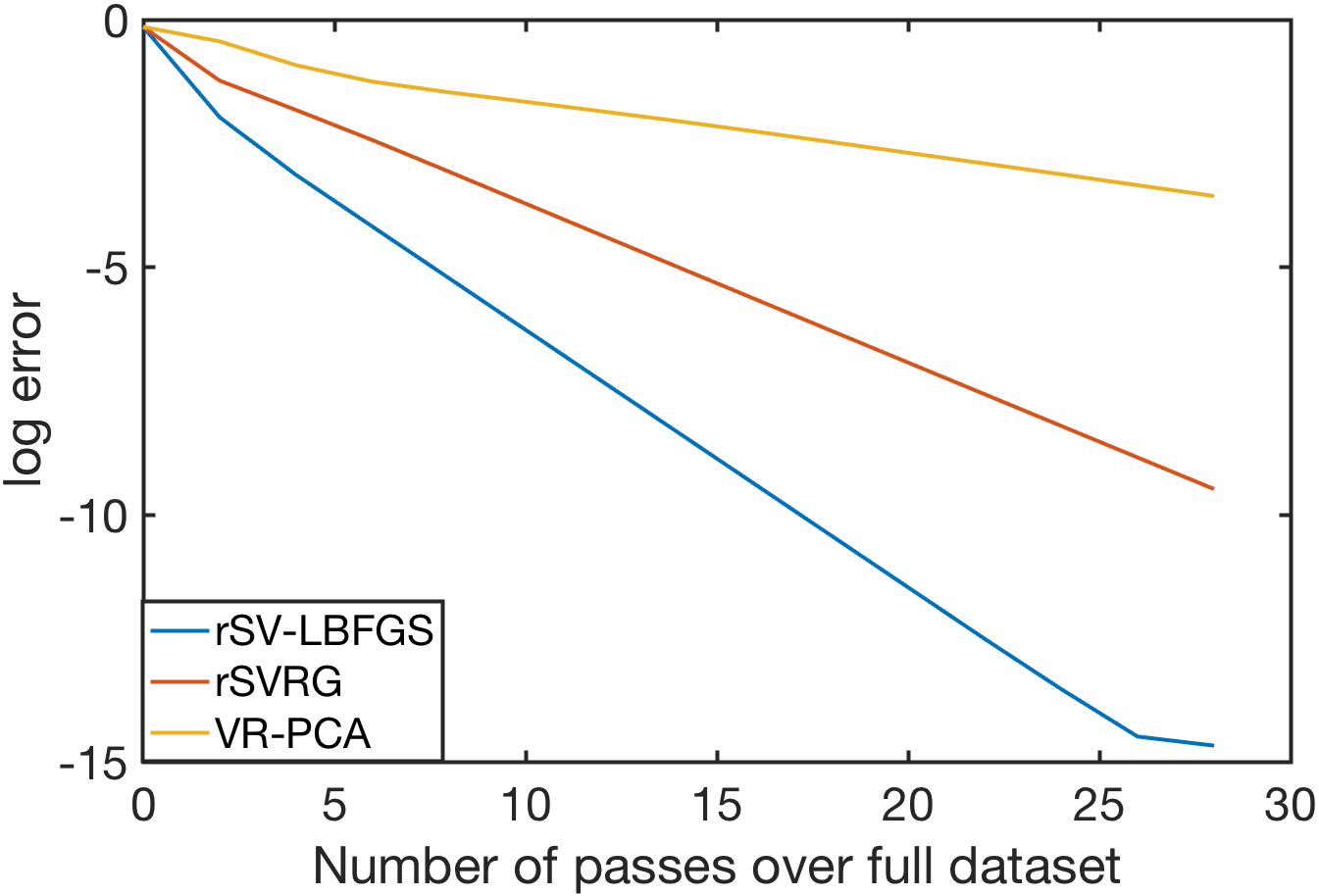}
\caption{ }
\label{fig:eigvdata3}
\end{subfigure}
\begin{subfigure}[b]{0.49\textwidth}
\includegraphics[width=\textwidth]{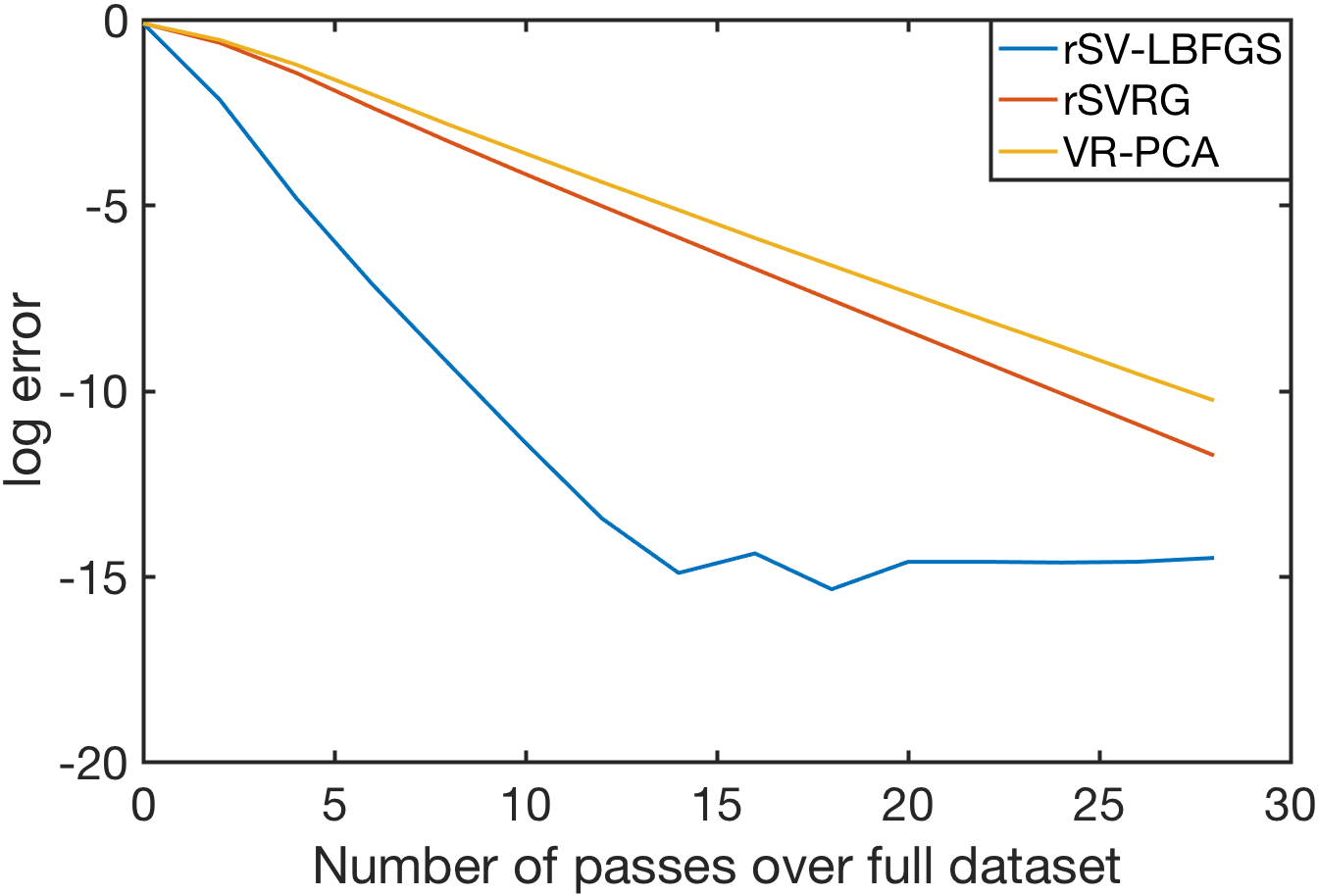}
\caption{ }
\label{fig:eigvdata4}
\end{subfigure}
\caption{Error decay plots for rSV-LBFGS, rSVRG and VR-PCA obtained for dominating eigenvalue computation on four synthetic data matrices. Eigengaps were $0.005$, $0.05$, $0.01$ and $0.1$ for (a), (b), (c) and (d) respectively. See text for experimental details.}
\label{fig:eigvsynth}
\end{figure*}

\subsection{Leading eigenvalue computation}

Next we conduct a synthetic experiment on calculating the leading eigenvalue of matrices. This of course is a common problem in machine learning, and as such is a unit-norm constrained nonconvex optimization problem in Euclidean space. It can be written as:
\begin{align*}
\min_{\mathbf{z}\in\mathbb{R}^{d}:\mathbf{z}^{T}\mathbf{z}=1} -\mathbf{z}^{T}\left(\frac{1}{N}\sum\limits_{i=1}^{N}d_{i}d_{i}^{T}\right)\mathbf{z},
\end{align*} 
where $D^{d\times N}$ is the data matrix, and $d_{i}$ are its columns. We can transform this into an unconstrained manifold optimization problem on the sphere defined by the norm constraint. To that end, we generated four sets of datapoints, for eigengaps $0.005$, $0.05$, $0.01$ and $0.1$, using the techniques described in \cite{shamiricml15}. Each dataset contains $100,000$ vectors of dimension $1000$. We used a minibatch size of $100$ for the two Riemannian algorithms. As before, learning rates for rSVRG were chosen according to the bounds in \cite{zhangnips16}. Selecting appropriate values for the four parameters in rSV-LBFGS (the first and second-order learning rates, $L$ and $M$) was a nontrivial task; after careful grid searches within the bounds defined by the convergence conditions, we chose $\eta_{1}=0.001, \eta_{2}=0.1,$ and $M=10$ for all four datasets. $L$ was set to $5$ for the dataset with eigengap $0.005$, and $10$ for the rest. The $z_{r}$ correction pair was calculated using \textbf{Option 2}: $z_{r}=\Gamma_{\gamma}\left(-\eta_{1}\nu_{\text{prev}}\right)$. We plot the performance of rSV-LBFGS, rSVRG and VR-PCA in Fig.\ref{fig:eigvsynth}. Extensive comparisons of VR-PCA with other Euclidean algorithms have been conducted in \cite{shamiricml15}; we do not repeat them here. We computed the error of iterate $\mathbf{z}$ as $1-\left(Ne^{*}\right)^{-1}\|D^{T}\mathbf{z}\|_{2}^{2}$, $e^{*}$ being the ground truth obtained from Matlab's $eigs$.

We see that the rSV-LBFGS method performs well on all four datasets, reaching errors of the order of $1e-15$ well before VR-PCA and rSVRG in the last three cases. The performance delta relative to VR-PCA is particularly noticeable in each of the four cases; we consider this to be a noteworthy result for fixed-stepsize algorithms on Riemannian manifolds.

\section{Conclusion} 
We propose a novel L-BFGS algorithm on Riemannian manifolds with variance reduced stochastic gradients, and provide theoretical analyses for strongly convex functions on manifolds. We conduct experiments on computing Riemannian centroids for symmetric positive definite matrices, and calculation of leading eigenvalues, both using large scale datasets. Our algorithm outperforms other Riemannian optimization algorithms with fixed stepsizes in both cases, and performs noticeably better than one of the fastest stochastic algorithms in Euclidean space, VR-PCA, for the latter case.

\bibliographystyle{unsrt}
\bibliography{refpaper}

\begin{thebibliography}{10}

\bibitem{absilbook}
P.~A. Absil, R.~Mahony, and R.~Sepulchre.
\newblock {\em Optimization {A}lgorithms on {M}atrix {M}anifolds}.
\newblock Princeton University Press, 2008.

\bibitem{zhangcolt16}
H.~Zhang and S.~Sra.
\newblock First-order methods for geodesically convex optimization.
\newblock In {\em COLT}, 2016.

\bibitem{conicsuvrit}
S.~Sra and R.~Hosseini.
\newblock Conic geometric optimization on the manifold of positive definite
  matrices.
\newblock {\em SIAM Journal on Optimization}, 25(1):713--739, 2015.

\bibitem{ring_wirth}
W.~Ring and B.~Wirth.
\newblock Optimization methods on {R}iemannian manifolds and their application
  to shape space.
\newblock {\em SIAM Journal on Optimization}, 22(2):596--627, 2012.

\bibitem{robbins_monro}
H.~Robbins and S.~Monro.
\newblock A {S}tochastic {A}pproximation {M}ethod.
\newblock {\em The Annals of Mathematical Statistics}, 22(3):400--407, 1951.

\bibitem{bonnabel}
S.~Bonnabel.
\newblock Stochastic gradient descent on {R}iemannian manifolds.
\newblock {\em IEEE Transactions on Automatic Control}, 58(9):2217--2229, 2013.

\bibitem{zhangnips16}
H.~Zhang, S.~J. Reddi, and S.~Sra.
\newblock Riemannian {SVRG}: Fast stochastic optimization on {R}iemannian
  manifolds.
\newblock In {\em NIPS}, 2016.

\bibitem{shamiricml15}
O.~Shamir.
\newblock A {S}tochastic {PCA} and {SVD} {A}lgorithm with an {E}xponential
  {C}onvergence {R}ate.
\newblock In {\em ICML}, 2015.

\bibitem{oja}
E.~Oja.
\newblock Principal components, minor components, and linear neural networks.
\newblock {\em Neural Networks}, 5(6):927--935, 1992.

\bibitem{rienips13}
R.~Johnson and T.~Zhang.
\newblock Accelerating {S}tochastic {G}radient {D}escent using {P}redictive
  {V}ariance {R}eduction.
\newblock In {\em NIPS}, 2013.

\bibitem{liulbfgs}
D.~C. Liu and J.~Nocedal.
\newblock On the limited memory {BFGS} method for large scale optimization.
\newblock {\em Mathematical Programming}, 45(1--3):503--528, 1989.

\bibitem{nocedal_wright}
J.~Nocedal and S.~J. Wright.
\newblock {\em Numerical {O}ptimization}.
\newblock Springer, 2006.

\bibitem{bottounips04}
L.~Bottou and Y.~LeCun.
\newblock Large scale online learning.
\newblock In {\em NIPS}, 2004.

\bibitem{bottouiccs10}
L.~Bottou.
\newblock Large-scale machine learning with stochastic gradient descent.
\newblock In {\em International Conference on Computational Statistics}, 2010.

\bibitem{byrdarxiv14}
R.~H. Byrd, S.~L. Hansen, J.~Nocedal, and Y.~Singer.
\newblock A {S}tochastic {Q}uasi-{N}ewton {M}ethod for {L}arge-scale
  {O}ptimization, 2014.
\newblock arXiv:1410.1068.

\bibitem{moritzaist16}
P.~Moritz, R.~Nishihara, and M.~I. Jordan.
\newblock A {L}inearly-{C}onvergent {S}tochastic {L-BFGS} {A}lgorithm.
\newblock In {\em AISTATS}, 2016.

\bibitem{leebook}
J.~Lee.
\newblock {\em Riemann {M}anifolds: an {I}ntroduction to {C}urvature}.
\newblock Springer-Verlag, 1997.

\bibitem{boothby}
W.~M. Boothby.
\newblock {\em An {I}ntroduction to {D}ifferentiable {M}anifolds and
  {R}iemannian {G}eometry}.
\newblock Academic Press Inc., 1986.

\bibitem{reshadnips15}
R.~Hosseini and S.~Sra.
\newblock Matrix {M}anifold {O}ptimization for {G}aussian {M}ixtures.
\newblock In {\em NIPS}, 2015.

\bibitem{reddinonconvex}
S.~J. Reddi, A.~Hefny, S.~Sra, B.~P\'{o}cz\'{o}s, and A.~Smola.
\newblock Stochastic {V}ariance {R}eduction for {N}onconvex {O}ptimization.
\newblock In {\em ICML}, 2016.

\bibitem{bhatiabook}
R.~Bhatia.
\newblock {\em Positive {D}efinite {M}atrices}.
\newblock Princeton University Press, 2007.

\bibitem{manopt}
N.~Boumal, B.~Mishra, P.-A. Absil, and R.~Sepulchre.
\newblock {M}anopt, a {M}atlab toolbox for optimization on manifolds.
\newblock {\em Journal of Machine Learning Research}, 15:1455--1459, 2014.

\bibitem{mixest}
R.~Hosseini and M.~Mash'al.
\newblock Mixest: {A}n {E}stimation {T}oolbox for {M}ixture {M}odels, 2015.
\newblock arXiv:1507.06065.

\bibitem{binietal}
D.~A. Bini and B.~Iannazzo.
\newblock Computing the {K}archer mean of symmetric positive definite matrices.
\newblock {\em Linear Algebra and its Applications}, 483(4):1700--1710, 2013.

\bibitem{mokhtarijmlr}
A.~Mokhtari and A.~Ribeiro.
\newblock Global {C}onvergence of {O}nline {L}imited {M}emory {BFGS}.
\newblock {\em Journal of {M}achine {L}earning {R}esearch}, 16(1):3151--3181,
  2015.

\end{thebibliography}

\newpage
\setcounter{proposition}{0}
\setcounter{lemma}{0}
\section{Appendices} 

We present the convergence results from Propositions 1 and 2 in the main text in this section. 

\subsection{Analysis of convergence} We analyze convergence for finite-sum empirical risk minimization problems of the following form:
\begin{equation}
\label{eq:fsum}
\min_{x\in \mathcal{M}} f(x) = \frac{1}{N}\sum\limits_{i=1}^{N}f_{i}(x),
\end{equation} where the Riemannian manifold is denoted by $\mathcal{M}$. 
Note that the iterates are updated in Algorithm \ref{alg:srlbfgs} by taking the exponential map of the descent step multiplied by the stepsize, with the descent step computed as the product of the inverse Hessian estimate and the stochastic variance-reduced gradient using the standard two-loop recursion formula. Thus, to bound the optimization error using the iterates, we will need bounds on both the stochastic gradients and the inverse Hessians. As mentioned in \cite{zhangnips16}, the methods used to derive the former bounds for Euclidean algorithms cannot be ported directly to manifolds due to metric nonlinearities; see the proof of Proposition \ref{prop:one} for details. For the latter, we follow the standard template for L-BFGS algorithms in the literature \cite{ring_wirth,nocedal_wright,byrdarxiv14}. To begin, we make the following assumptions: 

\textbf{ Assumption 1.} The function $f$ in \eqref{eq:fsum} is strongly convex on the manifold, whereas the $f_{i}$s are individually convex.

\textbf{ Assumption 2.} There exist $\lambda,\Lambda\in (0,\infty),\>\lambda<\Lambda$ s.t. $\lambda\|\nu\|_{x}^{2}\leq D^{2}f\leq \Lambda\|\nu\|_{x}^{2}\quad\forall\nu\in T_{x}\mathcal{M}$.

These two assumptions allow us to (a) guarantee that $f$ has a unique minimizer $x^{*}$ in the convex sublevel set $U$, and (b) derive bounds on the inverse Hessian updates using BFGS update formulae for the Hessian approximations. Similar to the Euclidean case, these can be written as follows:
\begin{equation}
\label{eqn:bfgsone}
\hat{B}_{r}=\Gamma_{\gamma}\left[\hat{B}_{r-1}-\frac{B_{r-1}(s_{r-1},\cdot)\hat{B}_{r-1}s_{r-1}}{B_{r-1}(s_{r-1},s_{r-1})}\right]\Gamma_{\gamma}^{-1},
\end{equation} and by the Sherman-Morrison-Woodbury lemma, that of the inverse:
\begin{equation*}
H_{r} = \Gamma_{\gamma}\left[G^{-1}H_{r-1}G+\frac{g_{x_{r-1}}(s_{r-1},\cdot)s_{r-1}}{y_{r-1}s_{r-1}}\right]\Gamma_{\gamma}^{-1},
\end{equation*} where $G = I-\frac{g_{x_{r-1}}(s_{r-1},\cdot)\hat{y}_{r-1}}{y_{r-1}s_{r-1}}$. The $\hat{B}_{r}$ is the Lax-Milgram representation of the Hessian \cite{ring_wirth}.

\subsubsection{Trace and determinant bounds}
To start off our convergence discussions for both convex and nonconvex cases, we derive bounds for the trace and determinants of the Hessian approximations, followed by those for their inverses. The techniques used to do so are straightforward ports of the Euclidean originals \cite{nocedal_wright}, with some minor modifications to account for differential geometric technicalities. Using the assumptions above, we can prove the following bounds \cite{ring_wirth}:
\begin{lemma}
\label{lem:trdet}
Let $B_{r}^{s+1}=\left(H_{r}^{s+1}\right)^{-1}$ be the approximation of the Hessian generated by Algorithm~\ref{alg:srlbfgs}, and $\hat{B}_{r}^{s+1}$ and $\hat{H}_{r}^{s+1}$ be the corresponding Lax-Milgram representations. Let $M$, the memory parameter, be the number of correction pairs used to update the inverse Hessian approximation. Then, under Assumptions 1 and 2, we have:
\begin{align*}
\mathrm{tr}(\hat{B}_{r}^{s+1})\leq \mathrm{tr}(\hat{B}_{0}^{s+1})+M\Lambda,\quad \det(\hat{B}_{r}^{s+1})&\geq \det(\hat{B}_{0}^{s+1})\frac{\lambda^{M}}{(\mathrm{tr}(\hat{B}_{0}^{s+1})+\Lambda M)^{M}}.
\end{align*} Also, $\gamma I\preceq \hat{H}_{r}^{s+1}\preceq \Gamma I $, for some $\Gamma\geq\gamma>0$.
\end{lemma}

\begin{proof}
For brevity of notation we temporarily drop the $(s+1)$ superscript. The proof for the Euclidean case \cite{byrdarxiv14,mokhtarijmlr} can be generalized to the Riemannian scenario in a straightforward way, as follows. Define the average Hessian $G_{r}$ by 
\begin{equation*}
G_{r}(\cdot,\cdot)=\int\limits_{0}^{1}D^{2}[f(tz_{r})](\cdot,\cdot)dt,
\end{equation*} such that $y_{r}=G_{r}(z_{r},\cdot)$. Then, it can be easily shown that $G_{r}$ satisfies the bounds in Assumption 2. Furthermore, we have the following useful inequalities
\begin{equation}
\label{eqn:useful}
\frac{y_{r}z_{r}}{\Sqnorm{z_{r}}}=\frac{G_{r}(s_{r},s_{r})}{\Sqnorm{z_{r}}}\geq\lambda,\qquad\frac{\Sqnorm{y_{r}}}{y_{r}z_{r}}\leq\Lambda.
\end{equation}
Let $\hat{y}_{r}$ be the Riesz representation of $y_{r}$. Recall that parallel transport is an isometry along the unique geodesics, which implies invariance of the trace operator.  Then using the L-BFGS update \eqref{eqn:bfgsone} and \eqref{eqn:useful}, we can bound the trace of the Lax-Milgram representation of the Hessian approximations as follows:
\begin{align*}
\mathrm{tr}(\hat{B}_{r})&=\mathrm{tr}(\Gamma_{\gamma}\hat{B}_{r-1}\Gamma_{\gamma}^{-1})-\frac{\|\Gamma_{\gamma}\hat{B_{r-1}}s_{r-1}\|^{2}}{B_{r-1}(s_{r-1},s_{r-1})}+\frac{\|\Gamma_{\gamma}\hat{y}_{r-1}\|^{2}}{y_{r-1}s_{r-1}} \\
&\leq\>\mathrm{tr}(\Gamma_{\gamma}\hat{B}_{r-1}\Gamma_{\gamma}^{-1})+\frac{\|\Gamma_{\gamma}\hat{y}_{r-1}\|^{2}}{y_{r-1}s_{r-1}} \\
&\leq\>\mathrm{tr}(B_{0})+M\Lambda.
\end{align*} This therefore proves boundedness of the largest eigenvalue of the $\hat{B}_{r}$ estimates.

Similarly, to get a lower bound for the minimum eigenvalue, we bound the determinant as follows:
\begin{align*}
\det(\hat{B}_{r})=\>&\det(\Gamma_{\gamma}B_{r-1}\Gamma_{\gamma}^{-1})\cdot\det\left(I-\frac{\hat{B}_{r-1}s_{r-1}s_{r-1}}{B_{r-1}(s_{r-1},s_{r-1})} +\hat{B}_{r-1}^{-1}\frac{y_{r-1}y_{r-1}}{y_{r-1}s_{r-1}}\right) \\
=\>&\det(\Gamma_{\gamma}B_{r-1}\Gamma_{\gamma}^{-1})\frac{y_{r-1}s_{r-1}}{B_{r-1}(s_{r-1},s_{r-1})} \\
=\>&\det(\Gamma_{\gamma}B_{r-1}\Gamma_{\gamma}^{-1})\frac{y_{r-1}s_{r-1}}{\|s_{r-1}\|^{2}}\cdot\frac{\|s_{r-1}\|^{2}}{B_{r-1}(s_{r-1},s_{r-1})}\\
\geq\>&\det(\Gamma_{\gamma}B_{r-1}\Gamma_{\gamma}^{-1})\frac{\lambda}{\lambda_{max}(B_{r-1})},
\end{align*} where we use $\lambda_{\text{max}}$ to denote the maximum eigenvalue of $B_{r-1}$, and use \eqref{eqn:useful}. Since $\lambda_{\text{max}}$ is bounded above by the trace of $(\hat{B}_{r})$, we can telescope the inequality above to get
\begin{equation*}
\det(\hat{B}_{r})\geq\>\det(B_{0})\frac{\lambda^{M}}{(\mathrm{tr}(B_{0})+M\Lambda)^{M}}.
\end{equation*} The bounds on the maximum and minimum eigenvalues of $B_{r}$ thus derived allows us to infer corresponding bounds for those of $H_{r}$ as well, since by definition $H_{r}=\hat{B}_{r}^{-1}$.
\end{proof}

\subsubsection{Convergence results for the strongly convex case} Next we provide a brief overview of the bounds necessary to prove our convergence result. First, note the following bound implied by the Lipschitz continuity of the gradients:
\begin{align*}
f(x_{t+1}^{s+1}) \leq f(x_{t}^{s+1}) +g(\nabla f(x_{t}^{s+1}), \Expmapinv{x_{t}^{s+1}}{x_{t+1}^{s+1}}) + \frac{L}{2}\Sqnorm{\Expmapinv{x_{t}^{s+1}}{x_{t+1}^{s+1}}}.
\end{align*} Note the update step fom line 11 of Algorithm~\ref{alg:srlbfgs}: $x_{t+1}^{s+1}=\Expmap{x_{t}^{s+1}}{-\eta H_{r}^{s+1}\nu_{t}^{s+1}}$. We can replace the inverse exponential map in the inner product above by the quantity in the parentheses. In order to replace $H_{r}^{s+1}$ by the eigen-bounds from Lemma~\ref{lem:trdet}, we invoke the following result (Lemma 5.8 from \cite{leebook}:
\begin{lemma}
For any $D\in\Tanspace$ and $c,t\in\mathbb{R}$, $\gamma_{cD}(t)=\gamma_{D}(ct)$,
\end{lemma} where $\nu$ is the ``speed'' of the geodesic. This allows us to write $\Expmap{x}{c\nu}=\gamma_{\nu}(c)=\gamma_{c\nu}(1)$. Recall that for Riemann geodesics we have $\|\dot{\gamma(t)}\|=\bar{s}\text{ for all }t\in[0,1]$, a constant.

\begin{proposition}
\label{prop:one}
Let the Assumptions 1 and 2 hold. Further, let the $f(\cdot)$ in \eqref{eq:fsum} be S-strongly convex, and each of the $f_{i}$ be L-smooth, as defined earlier. Define the following constants: $p=\left[LS^{-1}+2\eta_{2}S^{-1}\left\lbrace2\eta L^{3}\Gamma^{2}-S\kappa\gamma\right\rbrace\right]$, and $q^{\prime}=6\eta^{2}L^{3}\Gamma^{2}S^{-1}$. Denote the global optimum by $x^{*}$. Then the iterate $x^{T+1}$ obtained after $T$ outer loop iterations will satisfy the following condition:
\begin{align*}
\mathbb{E}\left[f(x^{T+1})-f(x^{*})\right]\leq LS^{-1}\beta^{T}\mathbb{E}\left[f(x^{0})-f(x^{*})\right],
\end{align*} where the constants are chosen to satisfy $\beta=\left(1-p\right)^{-1}\left(q^{\prime}+p^{T}(1-p-q^{\prime})\right)<1$ for linear convergence.
\end{proposition}

\begin{proof}
From the $L$-smoothness condition \eqref{eqn:lsmooth}, we have the following:
\begin{align*}
f(x_{i+1}^{t+1})&\leq f(x_{i}^{t+1})+g\left(\nabla f(x_{i}^{t+1})\cdot\Expmapinv{x_{i}^{t+1}}{x_{i+1}^{t+1}}\right)\frac{L}{2}\|\Expmapinv{x_{i}^{t+1}}{x_{i+1}^{t+1}}\|^{2} \\
&=f(x_{i}^{t+1})-\eta_{2}\cdot g\left(\nabla f(x_{i}^{t+1}), H_{r}^{t+1}\nu_{i}^{t+1}\right)+\frac{L\eta_{2}^{2}}{2}\|H_{r}^{t+1}\nu_{i}^{t+1}\|^{2},
\end{align*} where we have omitted subscripts from the metric. Taking expectations, and using the bounds on the inverse Hessian estimates derived in Lemma \ref{lem:trdet}, we have the following:
\begin{align}
\label{eqn:temp1}
\mathbb{E}f(x_{i+1}^{t+1})\leq \mathbb{E}f(x_{i}^{t+1})-\eta_{2}\gamma\|\nabla f(x_{i}^{t+1})\|^{2}+\eta_{2}^{2}L^{3}\Gamma^{2}\left[2\|\Expmapinv{x_{i}^{t+1}}{x^{*}}\|^{2}+3\|\Expmapinv{x^{t}}{x^{*}}\|^{2}\right],
\end{align} where we have used the following bound on the stochastic variance-reduced gradients derived in \cite{zhangnips16}:
\begin{align*}
\mathbb{E}\|\nu_{i}^{t+1}\|^{2}\leq 4L^{2}\|\Expmapinv{x_{i}^{t+1}}{x^{*}}\|^{2}+6L^{2}\|\Expmapinv{x^{t}}{x^{*}}\|^{2}.
\end{align*} This can be derived using triangle inequalities and the $L$-smoothness assumption. Note that the bound is different from the Euclidean case \cite{rienips13}, due to technicalities introduced by the Riemannian metric not being linear in general.

Now, recall the condition $f(x^{t})-f(x^{*})\leq 2\kappa)^{-1}\|\nabla f(x^{t})\|^{2}$, which follows from strong convexity. Using this, we can derive a bound on the $\nabla$ term in \eqref{eqn:temp1} as follows:
\begin{align*}
\|\nabla f(x_{i}^{t+1})\|^{2}\geq 2\kappa\left(f(x_{i}^{t+1})-f(x^{*})\right)\geq S\kappa\|\Expmapinv{x_{i}^{t+1}}{x^{*}}\|^{2},
\end{align*} where the second inequality follows from $S$-strong convexity \eqref{eqn:sconvex}, since $\nabla f(x^{*})=0$. Plugging this into \eqref{eqn:temp1}, we have the following:
\begin{align}
\label{eqn:temp2}
\begin{split}
\mathbb{E}f(x_{i+1}^{t+1})\leq f(x_{i}^{t+1})&+\eta_{2}\left[2\eta L^{3}\Gamma^{2}-S\kappa\gamma\right]\|\Expmapinv{x_{i}^{t+1}}{x^{*}}\|^{2} \\
&+3\eta_{2}^{2}L^{3}\Gamma^{2}\|\Expmapinv{x^{t}}{x^{*}}\|^{2}.
\end{split}
\end{align}
Now, note that $S$-strong convexity allows us to write the following:
\begin{align*}
\frac{S}{2}\|\Expmapinv{x_{i+1}^{t+1}}{x^{*}}\|^{2}&\leq f(x_{i+1}^{t+1})-f(x^{*}) \\
&=\left[f(x_{i+1}^{t+1})-f(x_{i}^{t+1})\right] + \left[f(x_{i}^{t+1}) -f(x^{*})\right] \\
&\leq\left[f(x_{i+1}^{t+1})-f(x_{i}^{t+1})\right] + \frac{L}{2}\|\Expmapinv{x_{i}^{t+1}}{x^{*}}\|^{2},
\end{align*} where the last step follows from $L$-smoothness. Taking expectations of both sides, and using \eqref{eqn:temp2} for the first component on the right, we have
\begin{align}
\label{eqn:temp3}
\begin{split}
\mathbb{E}\|\Expmapinv{x_{i+1}^{t+1}}{x^{*}}\|^{2}\leq &\left[\frac{L}{S}+\frac{2\eta_{2}}{S}\left\lbrace2\eta L^{3}\Gamma^{2}-S\kappa\gamma\right\rbrace\right]\|\Expmapinv{x_{i+1}^{t+1}}{x^{*}}\|^{2} \\
&+ \frac{6\eta_{2}^{2}L^{3}\Gamma^{2}}{S}\|\Expmapinv{x^{t}}{x^{*}}\|^{2}.
\end{split}
\end{align} Now, we denote $p=\left[\frac{L}{S}+\frac{2\eta_{2}}{S}\left\lbrace2\eta L^{3}\Gamma^{2}-S\kappa\gamma\right\rbrace\right]$, and $q^{\prime}= \frac{6\eta_{2}^{2}L^{3}\Gamma^{2}}{S}$. Then, taking expectations over the sigma algebra of all the random variables till minibatch $m$, and some algebra, it can be shown that:
\begin{align*}
\mathbb{E}\|\Expmapinv{x_{m}^{t+1}}{x^{*}}\|^{2}-q\mathbb{E}\|\Expmapinv{x_{0}^{t+1}}{x^{*}}\|^{2}\leq p^{m}\left(1-q\right)\|\Expmapinv{x_{0}^{t+1}}{x^{*}}\|^{2},
\end{align*} where $q=(1-p)^{-1}q^{\prime}$. Note that this provides a bound on the iterate at the end of the inner minibatch loop. Telescoping further, we have the bound
\begin{align*}
\mathbb{E}\|\Expmapinv{x^{T+1}}{x^{*}}\|^{2}\leq\beta^{T}\mathbb{E}\|\Expmapinv{x^{0}}{x^{*}}\|^{2},
\end{align*} where $\beta=\frac{q^{\prime}+p^{T}(1-p-q^{\prime})}{1-p}$. Then, using this result with a final appeal to the $L$-Lipschitz and $S$-strong convexity conditions, we have the bounds
\begin{align*}
\mathbb{E}\left[f(x^{T+1})-f(x^{*})\right]&\leq\frac{L}{2}\mathbb{E}\|\Expmapinv{x^{T+1}}{x^{*}}\|^{2} \\
&\leq\frac{L}{S}\beta^{T}\left[f(x^{0})-f(x^{*})\right],
\end{align*} thereby completing the proof.
\end{proof}

\subsubsection{Convergence results for the nonconvex case}
We begin with the following inequality involving the side lengths of a geodesic ``triangle'' \cite{conicsuvrit,bonnabel}:
\begin{lemma}
\label{lem:trig}
Let the sectional curvature of a Riemannian manifold be bounded below by $c_{\delta}$. Let $A$ be the angle between sides of length $b$ and $c$ in a triangle on the manifold, with the third side of length $a$, as usual. Then the following holds:
\begin{align*}
a^{2}\leq \frac{c\sqrt{|c_{\delta}|}}{\tanh\left(c\sqrt{|c_{\delta}|}\right)}b^{2}+c^{2}-2bc\cos A.
\end{align*}
\end{lemma}
The cosine is defined using inner products, as in the Euclidean case, and the distances using inverse exponential maps, as seen above. The following sequence of results and proofs are inspired by the basic structure of \cite{reddinonconvex}, with suitable modifications involving the inverse Hessian estimates from the L-BFGS updates.

\begin{lemma}
\label{lem:lyap}
Let the assumptions of Proposition 2 hold. Define the following functions:
\begin{align*}
c_{i}&=c_{i+1}\left(1+\beta\eta_{2}\Gamma+2\zeta L^{2}\eta_{2}^{2}\Gamma^{2}\right)+L^{3}\eta_{2}^{2}\Gamma^{2}, \\
\delta_{i}&=\eta_{2}\gamma - \frac{c_{i+1}\eta_{2}\Gamma}{\beta}-L\eta_{2}^{2}\Gamma-2c_{i+1}\zeta\eta_{2}^{2}\Gamma^{2}>0,
\end{align*} where $c_{i},c_{i+1},\beta,\eta_{2}>0$. Further, for $0\leq t\leq T-1$, define the Lyapunov function $R_{i}^{t+1}=\mathbb{E}\left[f(x_{i}^{t+1})+c_{i}\|\Expmapinv{x^{t}}{x_{i}^{t+1}}\|^{2}\right]$. Then we have the following bound:
\begin{align*}
\mathbb{E}\|\nabla f(x_{i}^{t+1})\|^{2}\leq \frac{R_{i}^{t+1}-R_{i+1}^{t+1}}{\delta_{i}}.
\end{align*}
\end{lemma}
\begin{proof}
As with the proof of Proposition 1, we begin with the following bound derived from $L$-smoothness:
\begin{align*}
\mathbb{E}f(x_{i+1}^{t+1})\leq\mathbb{E}\left[f(x_{i}^{t+1})-\eta_{2}\gamma\|\nabla f(x_{i}^{t+1})\|^{2}+\frac{L\eta_{2}^{2}\Gamma^{2}}{2}\|\nu_{i}^{t+1}\|^{2}\right],
\end{align*} where we have bound the bounds on the inverse Hessian derived earlier. Then using Lemma \ref{lem:trig} above, we have:
\begin{align*}
\mathbb{E}\|\Expmapinv{x^{t}}{x_{i+1}^{t+1}}\|^{2}&\leq\mathbb{E}\|\Expmapinv{x^{t}}{x_{i}^{t+1}}\|^{2}+\zeta\|\Expmapinv{x_{i}^{t+1}}{x_{i+1}^{t+1}}\|^{2}-2g\left(\Expmapinv{x_{i}^{t+1}}{x_{i+1}^{t+1}},\Expmapinv{x_{i}^{t+1}}{x^{*}}\right) \\
&\leq\mathbb{E}\left[\|\Expmapinv{x^{t}}{x_{i}^{t+1}}\|^{2}+\zeta\eta_{2}^{2}\Gamma^{2}\|\nu_{i}^{t+1}\|^{2}\right]\\
&\quad+2\eta_{2}\Gamma\left[(2\beta)^{-1}\|\nabla f(x_{i}^{t+1})\|^{2}+\frac{\beta}{2}\|\Expmapinv{x^{t}}{x_{i}^{t+1}}\|^{2}\right],
\end{align*} where we have used $g(a,b)\leq\frac{1}{2\beta}\|a\|^{2}+\frac{\beta}{2}\|b\|^{2}$. Note that we have used the norm of the inverse exponential maps as the side lengths in Lemma \ref{lem:trig}. Using these last two results, we can derive the following bound for the Lyapunov functions $R_{i+1}^{t+1}$:
\begin{align*}
R_{i+1}^{t+1}&\leq\mathbb{E}\left[f(x_{i}^{t+1})-\left\lbrace\eta_{2}\gamma-\frac{c_{i+1}\eta_{2}\Gamma}{\beta}\right\rbrace\|\nabla f(x_{i}^{t+1})\|^{2}\right]+\Gamma^{2}\left\lbrace c_{i+1}\zeta\eta_{2}^{2}+\frac{L\eta_{2}^{2}}{2}\right\rbrace\mathbb{E}\|\nu_{i}^{t+1}\|^{2} \\
&\quad+c_{i+1}\left\lbrace 1+\eta_{2}\Gamma\beta\right\rbrace\mathbb{E}\|\Expmapinv{x^{t}}{x_{i}^{t+1}}\|^{2}.
\end{align*} The norm of the stochastic variance reduced gradient can be bounded as follows \cite{zhangnips16,reddinonconvex}:
\begin{align*}
\mathbb{E}\|\nu_{i}^{t+1}\|^{2}\leq 2L^{2}\mathbb{E}\|\Expmapinv{x^{t}}{x_{i}^{t+1}}\|^{2}+2\mathbb{E}\|\nabla f(x_{i}^{t+1})\|^{2}.
\end{align*} This allows us to bound the Lyapunov function above as:
\begin{align*}
R_{i+1}^{t+1}\leq R_{i}^{t+1}-\left\lbrace\eta_{2}\gamma-\frac{c_{i+1}\eta_{2}\Gamma}{\beta}-L\eta_{2}^{2}\Gamma^{2}-2c_{i+1}\zeta\eta_{2}^{2}\Gamma^{2}\right\rbrace\mathbb{E}\|\nabla f(x_{i}^{t+1})\|^{2},
\end{align*} which completes the proof.
\end{proof}

Next we present a bound on $\|\nabla f(\cdot)\|^{2}$ using the $\delta_{i}$'s defined above (Thm 6 of \cite{zhangnips16}):
\begin{lemma}
\label{lem:lastbutone}
Let the conditions of Lemma \ref{lem:lyap} hold, and define the quantities therein. Let $\delta_{i}>0$ $\forall i\in[0,m]$, and $c_{m}=0$. Let $\delta_{\delta}=\min_{i}\delta_{i}$, and $K=mT$. Then if we randomly return one of the iterates $\left\lbrace x_{i}^{t+1}\right\rbrace_{i=1}^{m}$ as $x^{t+1}$, then:
\begin{align*}
\mathbb{E}\|\nabla f(x^{t})\|^{2}\leq \frac{f(x^{0})-f(x^{*})}{K\delta_{\delta}}.
\end{align*}
\end{lemma} This result can be shown by telescoping the bound derived in the previous lemma for the Lyapunov functions, using $c_{m}=0$.

\begin{proposition}
Let the sectional curvature of the manifold be bounded below by $c_{\delta}$, and the $f_{i}$ be $L$-smooth. Let $x^{*}$ be an optimum of $f(\cdot)$ in \eqref{eq:fsum}. Assume the eigenvalues of the inverse Hessian estimates are bounded. Set $\eta_{2}=\mu_{0}/\left(\Gamma Ln^{\alpha_{1}}\eta^{\alpha_{2}}\right)$, $K=mT$, and $m=\lfloor n^{\nicefrac{3\alpha_{1}}{2}}/\left(3\mu_{0}\zeta^{1-2\alpha_{2}}\right)\rfloor$, where $\alpha_{1}\in(0,1]$ and $\alpha_{2}\in[0,2]$. Then, for suitable choices of the inverse Hessian bounds $\gamma,\Gamma$, we can find values for the constants $\mu_{0}>0$ and $\epsilon>0$ so that the following holds:
\begin{align*}
\mathbb{E}\|\nabla f(x^{T})\|^{2}\leq(K\epsilon)^{-1}L\eta_{2}^{\alpha_{1}}\zeta^{\alpha_{2}}\left(f(x^{0})-f(x^{*})\right).
\end{align*}
\end{proposition}
\begin{proof}
We define $\beta=L\zeta^{1-\alpha_{2}}/\left(n^{\nicefrac{\alpha_{1}}{2}}\Gamma\right)$. Also, as mentioned in the proposition, $\eta_{2}=\mu_{0}/\left(\Gamma Ln^{\alpha_{1}}\eta^{\alpha_{2}}\right)$, with appropriate $\alpha_{1}$, $\alpha_{2}$. Note that we need a bound for $\delta_{\delta}$ to plug in the denominator of the bound in Lemma \ref{lem:lastbutone} above. This quantity can be lower bounded as follow:
\begin{align*}
\delta_{\delta}&=\min_{i}\delta_{i}\\
&=\min_{i}\left\lbrace \eta_{2}\gamma-\frac{c_{i+1}\eta_{2}\Gamma}{\beta}-L\eta_{2}^{2}\Gamma^{2}-2c_{i+1}\zeta^{2}\eta_{2}^{2}\Gamma^{2}\right\rbrace \\
&\geq\left\lbrace \eta_{2}\gamma-\frac{c_{0}\eta_{2}\Gamma}{\beta}-L\eta^{2}\Gamma^{2}-2c_{0}\zeta\eta_{2}^{2}\Gamma^{2}\right\rbrace.
\end{align*}
Now we need to bound $c_{0}$. To that end, telescoping the $c_{i+1}$ function defined in Lemma \ref{lem:lyap} above with $c_{m}=0$, and denoting $\theta=\eta_{2}\beta\Gamma+2\zeta\eta_{2}^{2}L^{2}\Gamma^{2}$, we get the following:
\begin{align*}
c_{0}=\frac{L\mu_{0}^{2}\left\lbrace\right(1+\theta)^m-1\rbrace}{n^{2\alpha_{1}}\zeta^{2\alpha_{2}}\theta}.
\end{align*} Using the definitions of $\eta_{2}$ and $\beta$ above, we note that $\theta <1/m$, implying $c_{0}\leq\frac{L\mu_{0}}{\zeta n^{\nicefrac{\alpha_{1}}{2}}}(e-1)$. Plugging this in the bound above, we posit that $\delta_{\delta}$ can be bounded below as follows:
\begin{align*}
\delta_{\delta}&\geq \eta_{2}\left\lbrace \gamma-\frac{\mu_{0}\Gamma(e-1)}{\zeta^{2-\alpha_{2}}}-\frac{\mu_{0}}{n^{\alpha_{2}}\zeta^{\alpha_{2}}}-\frac{2\mu_{0}^{2}(e-1)}{n^{\nicefrac{3\alpha_{1}}{2}}\zeta^{\alpha_{2}}}\right\rbrace \\
&\geq \frac{\epsilon}{Ln^{\alpha_{1}}\zeta^{\alpha_{2}}},
\end{align*} for some sufficiently small $\epsilon$, and suitable choices of the inverse Hessian bounds $\gamma,\Gamma$ and the rest of the parameters. Using this bound in the denominator of the right hand side of Lemma \ref{lem:lastbutone} above completes the proof.
\end{proof}

\subsection{Retractions}
We approximated the exponential maps with retractions from the Manopt \cite{manopt} toolbox. We used the following formulae: $\mathbb{R}_{x}(\eta\rho)=x\cos\|\eta\rho\|_{F}+\frac{\eta\rho}{\|\eta\rho\|_{F}}\sin\|\eta\rho\|_{F}$ for the sphere manifold, and $\mathbb{R}_{x}(\eta\rho)=x\cdot M_{x}(x \setminus \eta\rho)$ for the manifold of symmetric PD matrices, where $M_{x}$ denotes the matrix exponential, and $\setminus$ is matrix division. Here $x\in\mathcal{M}$ is some point on the manifold, $\rho\in T_{x}\mathcal{M}$ is some descent step evaluated at $x$, and $\eta$ is the stepsize. 

\end{document}